\documentclass[cmp-p]{svjour-p}
\usepackage{amsmath,amsfonts,amssymb}

\newcommand{\R}{{\mathbb R}}

\newcommand{\C}{{\mathbb C}}
\newcommand{\be}[1]{\begin{equation}\label{#1}}
\newcommand{\ee}{\end{equation}}
\newcommand{\ben}{\begin{eqnarray*}}
\newcommand{\een}{\end{eqnarray*}}
\renewcommand{\(}{\left(}
\renewcommand{\)}{\right)}

\newcommand{\iv}[1]{\int_{\R^3}{#1}\;dv}
\newcommand{\ix}[1]{\int_{\R^3}{#1}\;dx}
\newcommand{\ixv}[1]{\iint_{\R^3\times\R^3}{#1}\;dx\,dv}
\def\cprime{$'$}

\newcommand{\la}{\lambda}
\newcommand{\foral}{\hbox{ for all }}
\newcommand{\equ}[1]{\eqref{#1}}

\usepackage{color}

\definecolor{darkgreen}{rgb}{0,0.4,0}

\renewcommand{\qed}{\unskip\null\hfill$\square$\vskip 0.3cm}

\begin{document}

\title{Relative equilibria in continuous stellar dynamics}
\titlerunning{Multiple components configurations in continuous stellar dynamics}
\date{\today}

\author{Juan Campos, Manuel del Pino, and Jean Dolbeault}
\institute{\textsf{M. del Pino:} Departamento de Ingenier\'{\i}a Matem\'atica and CMM, Universidad de Chile, Casilla 170 Correo 3, Santiago, Chile.
\email{\textsf{delpino@dim.uchile.cl}}\\
\textsf{J. Campos, J. Dolbeault:} Ceremade (UMR CNRS n$^\circ$ 7534), Universit\'e Paris-Dauphine, Place de Lattre de Tassigny, 75775 Paris C\'edex~16, France.
\email{\textsf{campos@ceremade.dauphine.fr, dolbeaul@ceremade.dauphine.fr}}}

\date{\today}
\maketitle
\thispagestyle{empty}

\begin{abstract} We study a three dimensional continuous model of gravitating matter rotating at constant angular velocity. In the rotating reference frame, by a finite dimensional reduction, we prove the existence of non radial stationary solutions whose supports are made of an arbitrarily large number of disjoint compact sets, in the low angular velocity and large scale limit. At first order, the solutions behave like point particles, thus making the link with the \emph{relative equilibria} in $N$-body dynamics. \end{abstract}


\vspace*{6pt}\noindent{\bf Mathematical subject classification (2000).} Primary: 35J20; Secondary: 35B20, 35B38, 35J60, 35J70, 35Q35, 76P05, 76S05, 82B40 \bigskip
%
%
%
%
%
%
%
%
%

\section{Introduction and statement of the main results}\label{Sec:Intro}

We consider the Vlasov-Poisson system
\be{Syst:VP}
\left\{\begin{array}{l}
\displaystyle\partial_tf+v\cdot\nabla_xf-\nabla_x\phi\cdot\nabla_vf=0\\ \\
\displaystyle \phi=-\frac 1{4\pi\,|\cdot|}*\rho\;,\quad \rho :=\iv f
\end{array}\right.
\ee
which models the dynamics of a cloud of particles moving under the action of a mean field gravitational potential $\phi$ solving the Poisson equation: \hbox{$\Delta\phi=\rho$}. Kinetic models like system \equ{Syst:VP} are typically used to describe gaseous stars or globular clusters. Here $f=f(t,x,v)$ is the so-called \emph{distribution function,\/} a nonnegative function in $L^\infty(\R,L^1(\R^3\times\R^3))$ depending on \emph{time} $t\in\R$, \emph{position} $x\in\R^3$ and \emph{velocity} $v\in\R^3$, which represents a density of particles in the \emph{phase space,\/} $\R^3\times\R^3$. The function $\rho$ is the \emph{spatial density\/} function and depends only on $t$ and $x$. The total mass is conserved and hence
\[
\ixv{f(t,x,v)}=\ix{\rho(t,x)}=M
\]
does not depend on $t$.

\medskip The first equation in \equ{Syst:VP} is the \emph{Vlasov equation,\/} also known as the \emph{collisionless Boltzmann equation\/} in the astrophysical literature; see \cite{Binney-Tremaine}. It is obtained by writing that the mass is transported by the flow of Newton's equations, when the gravitational field is computed as a mean field potential. Reciprocally, the dynamics of discrete particle systems can be formally recovered by considering empirical distributions, namely measure valued solutions made of a sum of Dirac masses, and neglecting the self-consistent gravitational terms associated to the interaction of each Dirac mass with itself.

It is also possible to relate \equ{Syst:VP} with discrete systems as follows. Consider the case of $N$ gaseous spheres, far away one to each other, in such a way that they weakly interact through gravitation. In terms of system \equ{Syst:VP}, such a solution should be represented by a distribution function $f$, whose space density~$\rho$ is compactly supported, with several nearly spherical components. At large scale, the location of these spheres is governed at leading order by the $N$-body gravitational problem.

The purpose of this paper is to unveil this link by constructing a special class of solutions: we will build time-periodic, non radially symmetric solutions, which generalize to kinetic equations the notion of {\em relative equilibria} for the discrete $N$-body problem. Such solutions have a planar solid motion of rotation around an axis which contains the center of gravity of the system, so that the centrifugal force counter-balances the attraction due to gravitation. Let us give some details.

\medskip Consider $N$ point particles with masses $m_j$, located at points $x_j(t)\in\R^3$ and assume that their dynamics is governed by Newton's gravitational equations
\be{Eq:Newton}
m_j\,\frac{d^2x_j}{dt^2}=\sum_{j\neq k=1}^N\frac{m_j\,m_k}{4\pi}\,\frac{x_k-x_j}{|x_k-x_j|^3}\;,\quad j=1,\ldots N\;.
\ee
Let us write $x\in \R^3$ as $x= (x',x^3)\in\R^2\times\R\approx\C\times\R$ where, using complex notations, $x' = (x^1,x^2)\approx x^1 + i\,x^2$ and rewrite system \equ{Eq:Newton} in coordinates relative to a reference frame rotating at a constant velocity $\omega>0$ around the $x^3$-axis. This amounts to carry out the change of variables
\[
x= (e^{i\,\omega\,t}\,z',\,z^3), \quad z' = z^1+ i\,z^2\,.
\]
In terms of the coordinates $(z',z^3)$, system \equ{Eq:Newton} then reads
\be{poto}
\frac{d^2z_j}{dt^2}=\sum_{j\neq k=1}^N\frac{m_k}{4\pi}\,\frac{z_k-z_j}{|z_k-z_j|^3}+\omega^2\,(z_j',0)+\,2\,\omega\,\(i\,\tfrac{dz_j'}{dt},0\)\;,\quad j=1,\ldots N\;.
\ee
We consider solutions which are stationary in the rotating frame, namely constant solutions $(z_1,\ldots z_N)$ of system \equ{poto}. Clearly all $z_j$'s have their third component with the same value, which we assume zero. Hence, we have that
\[
z_k= (\xi_k, 0)\;,\quad \xi_k\in \C\;,
\]
where the $\xi_k$'s are constants and satisfy the system of equations
\be{ll}
\sum_{k\neq j=1}^N\frac{m_k}{4\pi}\,\frac{\xi_k-\xi_j}{|\xi_k-\xi_j|^3}+\,\omega^2\,\xi_j\;,\quad j=1,\ldots N\;.
\ee

In the original reference frame, the solution of~\equ{Eq:Newton} obeys to a rigid motion of rotation around the center of mass, with constant angular velocity $\omega$. This solution is known as a \emph{relative equilibrium}, thus taking the form
\[\label{xj}
x_j^\omega(t) = (e^{i\,\omega\,t}\,\xi_j,0 )\;,\quad \xi_j\in \C\;, \quad j=1,\ldots N\;.
\]
System \equ{ll} has a variational formulation. In fact a vector $(\xi_1,\ldots \xi_N)$ solves \equ{ll} if and only if it is a critical point of the function
\[
\mathcal V _m^\omega(\xi_1,\ldots \xi_N)\;:=\;\frac{1}{8\pi}\sum_{j\neq k=1}^N\frac{m_j\,m_k}{|\xi_k-\xi_j|}+ \frac{\omega^2}2\sum_{j=1}^Nm_j\,|\xi_j|^2\,.
\]
Here $m$ denotes $(m_j)_{j=1}^N$. A further simplification is achieved by considering the scaling
\be{Eqn:ScalingOmega}
\xi_j = \omega^{-2/3}\,\zeta_j\;,\quad\mathcal V _m^\omega(\xi_1,\ldots \xi_N) = \omega^{2/3}\,\mathcal V _m (\zeta_1,\ldots \zeta_N)
\ee
where
\[\label{V}
\mathcal V _m(\zeta_1,\ldots \zeta_N)\;:=\;\frac{1}{8\pi}\sum_{j\neq k=1}^N\frac{m_j\,m_k}{|\zeta_k- \zeta_j|}+ \frac{1}2\sum_{j=1}^Nm_j\,|\zeta_j|^2\,.
\]
This function has in general many critical points, which are all \emph{relative equilibria.} For instance, $\mathcal V_m$ clearly has a global minimum point.

\medskip Our aim is to construct solutions of gravitational models in continuum mechanics based on the theory of relative equilibria. We have the following result.
\begin{theorem}\label{teo1} Given masses $m_j$, $j=1,\ldots N$, and any sufficiently small $\omega >0$, there exists a solution $f_\omega (t,x,v)$ of equation \equ{Syst:VP} which is $\frac {2\pi}\omega$-periodic in time and whose spatial density takes the form
\[\label{rho}
\rho(t, x):= \iv{f_\omega}\,= \sum_{i=1}^N \rho_j( x- x_j^\omega(t) )\,+\,o(1)\;.
\]
Here $o(1)$ means that the remainder term uniformly converges to $0$ as \hbox{$\omega\to 0_+$} and identically vanishes away from $\cup_{j=1}^NB_R(x_j^\omega(t))$, for some $R>0$, independent of $\omega$. The functions $\rho_j(y)$ are non-negative, radially symmetric, non-increasing, compactly supported functions, independent of $\omega$, with $\int_{\R^3} \rho_j (y)\;dy\,=\,m_j$ and the points $x_j^\omega(t)$ are such that
\[\label{xj1}
x_j^\omega(t) = \omega^{-2/3}\,(e^{i\,\omega\,t}\,\zeta^\omega_j,0 )\;,\quad \zeta_j^\omega\in \C\;, \quad j=1,\ldots N
\]
and
\[
\lim_{\omega\to 0_+}{\mathcal V}_m (\zeta_1^\omega,\ldots \zeta_N^\omega)=\min_{ \C^N} {\mathcal V}_m\;, \quad \lim_{\omega\to 0_+}\nabla {\mathcal V}_m (\zeta_1^\omega,\ldots \zeta_N^\omega)=0\;.
\]
\end{theorem}
The solution of Theorem~\ref {teo1} has a spatial density which is nearly spherically symmetric on each component of its support and these ball-like components rotate at constant, very small, angular velocity around the $x^3$-axis. The radii of these balls are very small compared with their distance to the axis. We shall call such a solution a \emph{relative equilibrium} of \equ{Syst:VP}, by extension of the discrete notion. The construction provides much more accurate informations on the solution. In particular, the building blocks $\rho_j$ are obtained as minimizers of an explicit reduced free energy functional, under suitable mass constraints.

\medskip It is also natural to consider other discrete relative equilibria, namely critical points of the energy $\mathcal V_m$ that may or may not be globally minimizing, and ask whether associated relative equilibria of system \equ{Syst:VP} exist. There are plenty of relative equilibria of the $N$-body problem. For instance, if all masses $m_j$ are equal to some $m_*>0$, a critical point is found by locating the $\zeta_j$'s at the vertices of a regular polygon:
\be{polygon}
\zeta_j = r\,e^{2\,i\,\pi (j-1)/N}\;,\quad j=1,\ldots N\;,
\ee
where $r$ is such that
\[
\frac d{dr}\Big[\frac{a_N}{4\pi}\,\frac{m_*}r+\frac 12\,r^2\Big]=0\quad\mbox{with}\quad a_N:=\frac 1{\sqrt 2}\,\sum_{j=1}^{N-1}\frac 1{\sqrt{1-\cos\(2\pi j/N\)}}\;,
\]
\emph{i.e.} $r=\(a_N\,m_*/(4\pi)\)^{1/3}$. This configuration is called the \emph{Lagrange solution,\/} see~\cite{perkowalter}. The counterpart in terms of continuum mechanics goes as follows.
\begin{theorem}\label{Thm:Lagrange} Let $(\zeta_{1},\ldots\zeta_N)$ be a regular polygon, namely with $\zeta_j$ given by \equ{polygon}, and assume that all masses are equal. Then there exists a solution $f_\omega$ exactly as in Theorem \ref{teo1}, but with $\lim_{\omega \to 0_+}(\zeta_1^\omega,\ldots \zeta_N^\omega )=(\zeta_1,\ldots \zeta_N)$. \end{theorem}
Further examples of \emph{relative equilibria} in the $N$-body problem can be obtained for instance by setting $N-1$ points particles of the same mass at the vertices of a regular polygon centered at the origin, then adding one more point particle at the center (not necessarily with the same mass), and finally adjusting the radius. Another family of solutions, known as the \emph{Euler--Moulton solutions} is constituted by arrays of aligned points.

\medskip Critical points of the functional ${\mathcal V}_m$ are always degenerate because of their invariance under rotations: for any $\alpha\in \R$ we have
\[
{\mathcal V}_m(\zeta_1,\ldots \zeta_N) = {\mathcal V}_m( e^{i\,\alpha}\,\zeta_1,\ldots e^{i\,\alpha}\,\zeta_N)\;.
\]
Let $\bar \zeta= (\bar \zeta_1,\ldots \bar \zeta_N)$ be a critical point of ${\mathcal V}_m$ with $\bar \zeta_\ell \ne 0$. After a uniquely defined rotation, we may assume that $\bar \zeta_{\ell 2}=0$. Moreover, we have a critical point of the function of $2N-1$ real variables
\[
\tilde {\mathcal V}_m(\zeta_1,\ldots \zeta_{\ell 1},\ldots \zeta_N)\,:=\,{\mathcal V}_m(\zeta_1, \ldots (\zeta_{\ell 1},0),\ldots \zeta_N)\;.
\]
We shall say that a critical point of ${\mathcal V}_m$ is {\em non-degenerate up to rotations} if the matrix $D^2\tilde {\mathcal V}_m(\bar \zeta_1,\ldots \bar \zeta_{\ell 1},\ldots \bar \zeta_N)$ is non-singular. This property is clearly independent of the choice of $\ell$.

Palmore in \cite{MR0363076,MR0413647,MR0413646,MR0445539,MR0420713} has obtained classification results for the relative equilibria. In particular, it turns out that for {\em almost every choice of masses $m_j$,} all critical points of the functional ${\mathcal V}_m$ are non-degenerate up to rotations. Moreover, in such a case there exist at least $[2^{N-1}(N-2)+1]\,(N-2)\,!$ such distinct critical points. Many other results on \emph{relative equilibria} are available in the literature. We have collected some of them in \hbox{\ref{Appendix:RE}} with a list of relevant references. These results have a counterpart in terms of relative equilibria of system \equ{Syst:VP}.
\begin{theorem}\label{teo2} Let $(\zeta_{1},\ldots \zeta_N)$ be a non-degenerate critical point of ${\mathcal V}_m$ up to rotations. Then there exists a solution $f_\omega$ as in Theorem \ref{teo1}, which satisfies, as in Theorem~\ref{Thm:Lagrange}, $\lim_{\omega \to 0_+}(\zeta_1^\omega,\ldots \zeta_N^\omega )=(\zeta_1,\ldots \zeta_N)$. \end{theorem}
This paper is organized as follows. In the next section, we explain how the search for relative equilibria for the Vlasov-Poisson system can be reduced to the study of critical points of a functional acting on the gravitational potential. The construction of these critical points is detailed in Section~\ref{Sec:Results}. Sections \ref{Sec:Linear} and \ref{Sec:Projection} are respectively devoted to the linearization of the problem around a superposition of solutions of the problem with zero angular velocity, and to the existence of a solution of a nonlinear problem with appropriate orthogonality constraints depending on parameters $(\xi_j)_{j=1}^N$ related to the location of the $N$ components of the support of the spatial density. Solving the original problem amounts to make all corresponding Lagrange multipliers equal to zero, which is equivalent to find a critical point of a function depending on $(\xi_j)_{j=1}^N$: this is the variational reduction described in Section~\ref{Sec:Reduction}. The proof of Theorems~\ref{teo1}, \ref {Thm:Lagrange} and \ref{teo2} is given in Section~\ref{Sec:Proofs} while known results on relative equilibria for the $N$-body, discrete problem are summarized \hbox{in~\ref{Appendix:RE}}.

\section{The setup}\label{Sec:Setup}

Guided by the representation \equ{poto} of the $N$-body problem in a rotating frame,
we change variables in equation \equ{Syst:VP}, replacing $x=(x',x^3)$ and $v=(v',v^3)$ respectively by
\[
(e^{i\,\omega\,t}\,x', x^3)\quad \hbox{and}\quad (i\,\omega\,x'+ e^{i\,\omega\,t}\,v',v^3)\;.
\]
Written in these new coordinates, Problem \equ{Syst:VP} becomes
\be{Eqn:Kinetic}
\left\{\begin{array}{l}
\frac{\partial f}{\partial t}+v\cdot\nabla_x f-\nabla_xU\cdot\nabla_vf-\omega^2\,x'\cdot\nabla_{v'}f+2\,\omega\,i\,v'\cdot\nabla_{v'}f=0\;, \\ \\
\displaystyle U=-\frac 1{4\pi\,|\cdot|}*\rho\;,\quad \rho:= \iv f\;.
\end{array}\right.
\ee
The last two terms in the equation take into account the centrifugal and Coriolis force effects. System \equ{Eqn:Kinetic} can be regarded as the continuous version of problem~\equ{poto}. Accordingly, a {\em relative equilibrium} of System \equ{Syst:VP} will simply correspond to a stationary state of \equ{Eqn:Kinetic}.

Such stationary solutions of \equ{Eqn:Kinetic} can be found by considering for instance critical points of the \emph{free energy} functional
\[
\mathcal F[f]:=\ixv{\kern-4pt\beta(f)}+\frac 12\ixv{\kern-4pt\(|v|^2-\omega^2\,|x'|^2\)f}-\frac 12\ix{\kern-4pt|\nabla U|^2}
\]
for some arbitrary convex function $\beta$, under the mass constraint
\[
\ixv f=M\;.
\]
A typical example of such a function is
\be{PolytropicCase}
\beta(f)=\frac 1q\,\kappa_q^{q-1}\,f^q
\ee
for some $q\in(1,\infty)$ and some positive constant $\kappa_q$, to be fixed later. The corresponding solution is known as the solution of the \emph{polytropic gas model\/}, see \cite{bfh,MR965418,Binney-Tremaine,SaSo,MR1668556}.

\medskip When dealing with stationary solutions, it is not very difficult to rewrite the problem in terms of the potential. A critical point of $\mathcal F$ under the mass constraint $\ixv f=M$ is indeed given in terms of $U$ by
\be{Eq:f}
f(x,v)=\gamma\(\lambda+\tfrac 12\,|v|^2+ U(x)-\tfrac 12\,\omega^2\,|x'|^2\)
\ee
where $\gamma$ is, up to a sign, an appropriate generalized inverse of $\beta'$. In case \equ{PolytropicCase}, $\gamma(s)=\kappa_q^{-1}\,(-s)_+^{1/(q-1)}$, where $s_+=(s+|s|)/2$ denotes the positive part of~$s$. The parameter $\lambda$ stands for the Lagrange multiplier associated to the mass constraint, at least if $f$ has a single connected component. At this point, one should mention that the analysis is not exactly as simple as written above. Identity~\equ{Eq:f} indeed holds only component by component of the support of the solution, if this support has more than one connected component, and the Lagrange multipliers have to be defined for each component. The fact that
\[
U(x)\mathop{\sim}_{|x|\to\infty}-\frac{M}{4\pi\,|x|}
\]
is dominated by $-\frac 12\,\omega^2\,|x'|^2$ as $|x'|\to\infty$ is also a serious cause of trouble, which clearly discards the possibility that the free energy functional can be bounded from below if $\omega\neq 0$. This issue has been studied in \cite{DF2007}, in the case of the so-called \emph{flat systems.\/}

Finding a stationary solution in the rotating frame amounts to solving a non-linear Poisson equation, namely
\be{pb}
\Delta U= g\(\lambda+ U(x)-\tfrac 12\,\omega^2\,|x'|^2\)\quad\mbox{if}\;x\in\mbox{supp}(\rho)
\ee
and $\Delta U=0$ otherwise, where $g$ is defined by
\[
g(\mu):=\iv{\gamma\(\mu+\tfrac 12\,|v|^2\)}\;.
\]
Hence, the problem can also be reduced to look for a critical point of the functional
\[
\mathcal J[U]:=\frac 12\ix{|\nabla U|^2}+\ix{G\(\lambda+ U(x)-\tfrac 12\,\omega^2\,|x'|^2\)}-\ix{\lambda\,\rho}\;,
\]
where $\lambda=\lambda[x,U]$ is now a functional which is constant with respect to $x$, with value $\lambda_i$, on each connected component $K_i$ of the support of $\rho(x)=g\(\lambda[x,U]+ U(x)-\frac 12\,\omega^2\,|x'|^2\)$, and implicitly determined by the condition
\[
\int_{K_i}g\(\lambda_i+ U(x)-\tfrac 12\,\omega^2\,|x'|^2\)\,dx=m_i\;.
\]
By $G$, we denote a primitive of $g$ and the total mass is $M=\sum_{i=1}^Nm_i$. Hence we can rewrite $\mathcal J$ as
\[
\mathcal J[U]=\frac 12\ix{|\nabla U|^2}+\sum_{i=1}^N\left[\,\int_{K_i}{\kern -3pt G\(\lambda_i+ U(x)-\tfrac 12\,\omega^2\,|x'|^2\)}\;dx-m_i\,\lambda_i\right]\,.
\]

We may also observe that critical points of $\mathcal F$ correspond to critical points of the \emph{reduced free energy} functional
\[
\mathcal G[\rho]:=\ix{\Big(h(\rho)-\frac 12\,\omega^2\,|x'|^2\,\rho\Big)}-\frac 12\ix{|\nabla U|^2}
\]
acting on the spatial densities if $h(\rho)=\int_0^\rho g^{-1}(-s)\,ds$. Also notice that, using the same function~$\gamma$ as in \equ{Eq:f}, to each distribution function $f$, we can associate a \emph{local equilibrium}, or \emph{local Gibbs state},
\[
G_f(x,v)=\gamma\(\mu(\rho(x))+\tfrac 12\,|v|^2\)
\]
where $\mu$ is such that $g(\mu)=\rho$. This identity defines $\mu=\mu(\rho)=g^{-1}(\rho)$ as a function of $\rho$. Furthermore, by convexity, it follows that $\mathcal F[f]\ge\mathcal F[G_f]=\mathcal G[\rho]$ if $\rho(x)=\iv{f(x,v)}$, with equality if $f$ is a local Gibbs state. See \cite{DMOS} for more details.

\medskip Summarizing, the heuristics are now as follows. The various components $K_i$ of the support of the spatial density $\rho$ of a critical point are assumed to be far away from each other so that the dynamics of their center of mass is described by the $N$-body point particles system, at first order. On each component $K_i$, the solution is a perturbation of an isolated minimizer of the free energy functional~$\mathcal F$ (without angular rotation) under the constraint that the mass is equal to $m_i$. In the spatial density picture, on $K_i$, the solution is a perturbation of a minimizer of the reduced free energy functional~$\mathcal G$.

To further simplify the presentation of our results, we shall focus on the model of \emph{polytropic gases\/} corresponding to \equ{PolytropicCase}. In such a case, with $p:=\frac 1{q-1}+\frac 32$, $g$ is given by
\[
g(\mu)=(-\mu)_+^p
\]
if the constant $\kappa_q$ is fixed so that $\kappa_q={\scriptstyle 4\pi\sqrt 2\int_0^{+\infty}\!\sqrt t\,(1+t)^\frac 1{q-1}\,dt}={\scriptstyle (2\pi)^\frac 32\tfrac{\Gamma(\frac {q}{q-1})}{\Gamma(\frac32+\frac {q}{q-1})}}$.

\medskip Free energy functionals have been very much studied over the last years, not only to characterize special stationary states, but also because they provide a framework to deal with \emph{orbital stability,\/} which is a fundamental issue in the mechanics of gravitation. The use of a free energy functional, whose entropy part, $\ixv{\beta(f)}$ is sometimes also called the \emph{Casimir energy functional,\/} goes back to the work of V.I. Arnold (see~ \cite{MR0180051,MR0205552,MR1398637}). The variational characterization of special stationary solutions and their orbital stability have been studied by Y.~Guo and G. Rein in a series of papers \cite{MR1757074,MR1709211,MR1838751,Guo-Rein,MR1706880,MR1884728,MR1991989,GERHARDREIN06012005} and by many other authors, see for instance \cite{DF2007,0523,MR2164685,MR2424993,MR2425179,MR2390311,SaSo,MR2048565,MR1668556}.

The main drawback of such approaches is that stationary solutions which are characterized by these techniques are in some sense trivial: radial, with a single simply connected component support. Here we use a different approach to construct the solutions, which goes back to \cite{MR867665} in the context of Schr\"odinger equations. We are not aware of attempts to use dimensional reduction coupled to power-law non-linearities and Poisson force fields except in the similar case of a nonlinear Schr\"odinger equation with power law nonlinearity and repulsive Coulomb forces (see \cite{MR2232437}), or in the case of an attractive Hartree-Fock model (see \cite{krieger-2008}). Technically, our results turn out to be closely related to the ones in \cite{MR2119987,MR2180307}.

Compared to previous results on gravitational systems, the main interest of our approach is to provide a much richer set of solutions, which is definitely of interest in astrophysics for describing complex patterns like binary gaseous stars or even more complex objects. The need of such an improvement was pointed for instance in~\cite{GR2006}. An earlier attempt in this direction has been done in the framework of Wasserstein's distance and mass transport theory in \cite{MR2219334}. The point of this paper is that we can take advantage of the knowledge of special solutions of the $N$-body problem to produce solutions of the corresponding problem in continuum mechanics, which are still reminiscent of the discrete system.

\section{Construction of relative equilibria}\label{Sec:Results}

\subsection{Some notations}

We denote by $x=(x',x^3)\in\R^2\times\R$ a generic point in~$\R^3$. We may reformulate Problem \equ{pb}
in terms of the potential $u= -\,U$ as follows. Given $N$ positive numbers $\la_1,\ldots\la_N$ and a small positive parameter $\omega$, we consider the problem of finding $N$~non-empty, compact, disjoint, connected subsets~ $K_i$ of $\R^3$, $i=1,\,2\ldots N$, and a positive solution $u$ of the problem
\be{1}
-\Delta u = \sum_{i=1}^N\rho_i\quad\mbox{in}\;\R^3\,,\quad\rho_i:=\( u- \la_i + \tfrac12\,\omega^2\,|x'|^2\)_+^p\,\chi_i\;,
\ee
\be{2}
\lim_{|x|\to\infty}u(x)=0\;,
\ee
where $\chi_i$ denotes the characteristic function of $K_i$. We define the mass and the center of mass associated to each component by
\[
m_i^\omega:=\ix{\rho_i}\quad\mbox{and}\quad x_i^\omega:=\frac 1{m_i}\ix{x\,\rho_i}\;.
\]
We shall find a solution of \equ{1} as a critical point of $\mathsf J\,[u]=\mathcal J[-u]$ in case of~\equ{PolytropicCase}, with the notations of Section~\ref{Sec:Setup}, that is when $G(-s)=\tfrac 1{p+1}\,s_+^{p+1}$ for any $s\in\R$. The functional $\mathsf J$ reads
\be{J}
\mathsf J\,[u] = \frac 12 \ix{|\nabla u|^2} - \frac 1{p+1}\sum_{i=1}^N \ix{\big( u-\la_i + \tfrac12\,\omega^2\,|x'|^2\big)^{p+1}_+\,\chi_i}\;.
\ee

Heuristically, our method goes as follows. We first consider the so-called \emph{basic cell} problem: we characterize the solution with a single component support, when $\omega=0$ and then build an \emph{ansatz} by considering approximate solutions made of the superposition of basic cell solutions located close to relative equilibrium points, when they are far apart from each other. This can be done using the scaling invariance, in the low angular velocity limit $\omega\to 0_+$. The proof of our main results will be given in Sections~\ref {Sec:Linear}--\ref{Sec:Proofs}. It relies on a dimensional reduction of the variational problem: we shall prove that for a well chosen $u$, $\mathsf J\,[u]=\sum_{i=1}^N \la_i^{5-p}\,\mathsf e_*\,-\,\mathcal V_m^\omega (\xi_1,\ldots\xi_N)\,+\,o(1)$ for some constant $\mathsf e_*$, up to $o(1)$ terms, which are uniform in $\omega>0$, small. Hence finding a critical point for $\mathsf J$ will be reduced to look for a critical point of $\mathcal V_m^\omega$ as a function of $(\xi_1,\ldots\xi_N)$.

\subsection{The basic cell problem}

Let us consider the following problem
\be{Pb1}
-\Delta w = (w-1)_+^p \quad \hbox{in}\;\R^3\,.
\ee
\begin{lemma}\label{Lem:Cell} Under the condition $\lim_{|x|\to\infty}w(x)=0$, Equation~\equ{Pb1} has a unique solution, up to translations, which is positive and radially symmetric. \end{lemma}
\begin{proof} Since $p$ is subcritical, it is well known that the problem
\[
-\Delta Z = Z^p \quad \hbox{in}\;B_1(0)
\]
with homogeneous Dirichlet boundary conditions, $Z=0$, on $\partial B_1(0)$, has a unique positive solution, which is also radially symmetric (see \cite{MR544879}). For any $R>0$, the function $Z_R (x) := R^{-2/(p-1)}\,Z(x/R)$ is the unique radial, positive solution of
\[
-\Delta Z_R = Z_R^p \quad \hbox{in}\;B_R(0)
\]
with homogeneous Dirichlet boundary conditions on $\partial B_R(0)$. According to \cite{flucherwei,MR544879}, any positive solution of \equ{Pb1} is radially symmetric, up to translations. Finding such a solution $w$ of \equ{Pb1} is equivalent to finding numbers $R>0$ and $m_*>0$ such that the function, defined by pieces as $w=Z_R +1$ in $B_R$ and $w(x) = m_*/(4\pi\,|x|)$ for any $x\in\R^3$ such that $|x|>R$, is of class $C^1$. These numbers are therefore uniquely determined by
\[
w(R^-)= 1= \frac{m_*}{4\pi\,R} = w(R^+)\,,\;w'(R^-)= R^{-\frac 2{p-1}-1}\,Z' (1) = -\frac{m_*}{4\pi\,R^2} = w'(R^+)\,,
\]
which uniquely determines the solution of \equ{Pb1}.
\qed\end{proof}

Now let us consider the slightly more general problem
\[
-\Delta w^\la = (w^\la-\la )_+^p \quad \hbox{in}\;\R^3
\]
with $\lim_{|x|\to\infty}w^\la(x)=0$. For any $\lambda>0$, it is straightforward to check that it has a unique radial solution given by
\[
w^\la (x) = \la\,w\big( \la^{(p-1)/2}\,x\big)\quad\forall\;x\in\R^3\,.
\]
Let us observe, for later reference, that
\be{ppp}
\ix{(w^\la-\la )_+^p} = \la^{(3-p)/2} \ix{(w -1 )_+^p}=:\la^{(3-p)/2}\,m_*\,.
\ee
Moreover, $w^\la$ is given by
\[
w^\la (x) = \frac{m_*}{4\pi\,|x|}\,\la^{(3-p)/2}\quad\forall\;x\in\R^3\;\mbox{such that}\;|x|>R\,\lambda^{-(p-1)/2}\,.
\]

\subsection{The ansatz}

We consider now a first approximation of a solution of \equ{1}-\equ{2}, built as a superposition of the radially symmetric functions $w^{\la_i}$ translated to points $\xi_i$, $i=1,\ldots N$ in~$\R^2\times\{0\}$, far away from each other:
\[
w_i(x) := w^{\la_i} (x-\xi_i)\;, \quad W_\xi: = \sum_{i=1}^N w_i\;.
\]
Recall that we are given the masses $m_1,\ldots m_N$. We choose, according to formula~\equ{ppp}, the positive numbers $\la_i$
so that
\[
\ix{(w_i -\la_i)^p_+}\;=\;m_i\quad \foral i=1,\ldots N\;.
\]
By $\xi$ we denote the array $(\xi_1,\,\xi_2,\ldots\xi_N)$.

We shall assume in what follows that the points~$\xi_i$ are such that for a large, fixed $\mu>0$, and all small $\omega>0$ we have
\be{constraints}
|\xi_i|<\mu\,\omega^{-2/3}\,,\quad |\xi_i-\xi_j|>\mu^{-1}\,\omega^{-2/3}\;.
\ee
Equivalently,
\be{constraints1}
|\zeta_i|<\mu\,\,,\quad |\zeta_i-\zeta_j|>\mu^{-1}\quad\hbox{ where }\quad \zeta_i := \omega^{-2/3}\,\xi_i\;.
\ee

We look for a solution of \equ{1} of the form
\[
u = W_\xi+ \phi
\]
for a convenient choice of the points $\xi_i$, where $\phi$ is globally uniformly small when compared with $W_\xi$. For this purpose, we consider a fixed number $R>1$ such that
\[
\mbox{\rm supp}\,(w^{\la_i}-\la_i )_+\subset B_{R-1}(0)\quad\forall\;i=1,\,2\ldots N
\]
and define the functions
\[
\chi (x) = \left\{ \begin{matrix} 1 & \hbox{ if } |x| < R \\ &\\ 0 & \hbox{ if } |x| \ge R \end{matrix} \right. \quad\mbox{and}\quad \chi_i(x) = \chi (x-\xi_i )\;.
\]
Thus we want to find a solution to the problem
\[
\Delta ( W_\xi+ \phi) + \sum_{i=1}^N \(W_\xi- \la_i + \phi + \tfrac12\,\omega^2\,|x'|^2 \)^p_+\,\chi_i = 0\quad\hbox{in}\;\R^3
\]
with $\lim_{|x|\to\infty}\phi(x)=0$, that is we want to solve the problem
\be{ddd}
\Delta \phi + \sum_{i=1}^N p\(W_\xi-\la_i + \tfrac12\,\omega^2\,|x'|^2\)^{p-1}_+ \chi_i\,\phi = -\mathsf E - \mathsf N[\phi]
\ee
where
\begin{eqnarray*}
&&\mathsf E := \Delta W_\xi + \sum_{i=1}^N \(W_\xi- \la_i + \tfrac12\,\omega^2\,|x'|^2 \)^p_+\,\chi_i\;,\\
&&\mathsf N[\phi] := \sum_{i=1}^N \Big[ \(W_\xi- \la_i + \tfrac12\,\omega^2\,|x'|^2 + \phi \)^p_+ - \(W_\xi- \la_i + \tfrac12\,\omega^2\,|x'|^2 \)^p_+ \\
&&\hspace*{6cm}- p\(W_\xi-\la_i + \tfrac12\,\omega^2\,|x'|^2\)^{p-1}_+\phi\,\Big]\,\chi_i\;.
\end{eqnarray*}

\section{A linear theory}\label{Sec:Linear}

The purpose of this section is to develop a solvability theory for the operator
\[
\mathsf L[\phi] = \Delta \phi + \sum_{i=1}^N p\(W_\xi-\la_i + \tfrac12\,\omega^2\,|x'|^2\)^{p-1}_+ \chi_i\,\phi\;.
\]
To this end we introduce the norms
\[
\| \phi\|_* = \sup_{x\in \R^3}\!\left (\sum_{i=1}^N |x-\xi_i | + 1\!\right )\!|\phi(x) |\;,\quad \| h \|_{**} = \sup_{x\in \R^3}\!\left (\sum_{i=1}^N |x-\xi_i |^4 + 1\!\right )\!|h(x)|\;.
\]
We want to solve problems of the form $\mathsf L[\phi] = h$ with $h$ and $\phi$ having the above norms finite. Rather than solving this problem directly, we consider a {\em projected problem} of the form
\be{l1}
\mathsf L[\phi] = h + \sum_{i=1}^N\sum_{j=1}^3 c_{ij}\,Z_{ij}\,\chi_i\;,
\ee
\be{l2}
\lim_{|x|\to\infty}\phi(x)=0\;,
\ee
where $Z_{ij}:= \partial_{x_j}w_i$, subject to orthogonality conditions
\be{l3}
\ix{\phi\,Z_{ij}\,\chi_i} = 0 \quad\forall\;i=1,\,2\ldots N,\;j=1,\,2,\,3\;.
\ee
Equation \equ{l1} involves the coefficients $c_{ij}$ as Lagrange multipliers associated to the constraints \equ{l3}. If we can solve the equations $\mathsf L[\psi]=h$ and $\mathsf L[Y_{ij}]=Z_{ij}$, and if we define
$c_{ij}$ such that $\int_{\R^3}\psi\,Z_{ij}\,dx+c_{ij}\int_{\R^3}Y_{ij}\,Z_{ij}\,dx=0$, then we observe that $\phi=\psi+\sum_{i,\,j}c_{ij}\,Y_{ij}$ solves \equ{l1} and satisfies \equ{l3}. However, for existence, we will rather reformulate the question as a constrained variational problem; see Equation~\equ{Eq:LinearVar} below.
\begin{lemma}\label{linear} Assume that \equ{constraints} holds. Given $h$ with $\|h\|_{**} < +\infty $, Problem \equ{l1}-\equ{l3} has a unique solution $\phi=:\mathsf T[h]$ and there exists a positive constant $C$, which is independent of $\xi$ such that, for $\omega>0$ small enough,
\be{cota}
\|\phi\|_* \le C\,\|h\|_{**}\;.
\ee\end{lemma}
\begin{proof} In order to solve \equ{l1}-\equ{l3}, we first establish \equ{cota} as an \emph{a priori\/} estimate. Assume by contradiction the existence of sequences $\omega_n\to 0$, $\xi_i^n$ satisfying \equ{constraints} for $\omega = \omega_n$, of functions $\phi_n$, $h_n$ and of constants $c_{ij}^n$ for which
\[
\begin{array}{c}
\displaystyle\|\phi_n\|_* = 1\;,\quad\lim_{n\to\infty}\|h_n\|_{**}=0\;,\\
\displaystyle\ix{\phi_n\,Z_{ij}\,\chi_i} = 0\quad\forall\;i\,,\;j\quad\mbox{and}\quad L[\phi_n] = h_n + \sum_{i=1}^N\sum_{j=1}^3 c^n_{ij}\,Z_{ij}\,\chi_i\;.
\end{array}
\]
Testing the equation against $Z_{kl}$, we obtain, after an integration by parts,
\[
O(\omega_n^{2/3})\,\| \phi_n\|_{*} = \ix{h_n\,Z_{kl}} + c^n_{kl} \ix{|Z_{kl}|^2\,\chi_k} + O(\omega_n^{4/3}) \sum_{(i,j)\ne (k,\,l)} |c_{ij}^n|\;.
\]
{}From these relations we then get $c^n_{kl} =O(\omega_n^{2/3})+O(\|h_n\|_{**})\to 0$ for all $k$, $l$.

Let us prove that $\lim_{n\to\infty}\|\phi_n\|_{L^\infty(\R^3)}=0$. If not, since $\|\phi_n\|_* =1$, we may assume that there is an index $i$ and a sufficiently large number $R>0$ for which
\[
\liminf_{n\to\infty}\|\phi_n\|_{L^\infty (B_R(\xi_i))} > 0\;.
\]
Using elliptic estimates, and defining $\psi_n (x) = \phi_n(\xi_i^n +x)$, we may assume that~$\psi_n$ uniformly converges in $C^1$ sense over compact subsets of $\R^3$ to a bounded, non-trivial solution $\psi$ of the equation
\[
\begin{array}{c}
\displaystyle
\Delta \psi + p\,\big(w^{\la_i} -\la_i\big)_+^{p-1} \psi = 0\;,\\
\displaystyle\ix{\psi\,\partial_{x_j} w^{\la_i}\,\chi} = 0 \quad \forall\;j=1,\,2,\,3\;.
\end{array}
\]
According to \cite[Lemma 5]{flucherwei}, $\psi$ must be a linear combination of the functions $\partial_{x_j} w^{\la_i}$, $j=1,2,3$. The latter orthogonality conditions yield $\psi \equiv 0$. This is a contradiction and the claim is proven. Finally, let
\[
\tilde h_n := h_n + \sum_{i=1}^N\sum_{j=1}^3 c^n_{ij}\,Z_{ij}\,\chi_i\;.
\]
Then we have that
\[
|\tilde h_n(x)| \le \Big(O(\omega_n^{2/3})+O(\|h_n\|_{**})\Big) \sum_{i=1}^k \frac 1 {1+ |x-\xi_i^n|^4}
\]
and hence $\tilde \phi_n$, the unique solution in $\R^3$ of
\[
\Delta \tilde\phi_n = \tilde h_n\;, \quad \lim_{|x|\to \infty } \tilde\phi_n(x) = 0\;,
\]
satisfies
\[
|\tilde \phi_n(x)| \le \Big(O(\omega_n^{2/3})+O(\|h_n\|_{**})\Big) \sum_{i=1}^k\frac 1 { |x-\xi_i^n|}\;.
\]
Now, since $\phi_n-\tilde \phi_n$ is harmonic in $\R^3\setminus \cup_i B_R(\xi_i^n)$, it tends to zero as $|x|\to \infty$ and gets uniformly small on the boundary of this set. By the maximum principle, we get the estimate
\[
|\phi_n(x)| \le \Big(O(\omega_n^{2/3})+O(\|h_n\|_{**})\Big) \sum_{i=1}^k \frac 1 { |x-\xi_i^n|}\quad\hbox{on }\R^3\setminus \cup_i B_R(\xi_i^n)\;.
\]
This shows that $\lim_{n\to\infty}\|\phi_n\|_* = 0 $, a contradiction with $\|\phi_n\|_* =1$, and \equ{cota} follows.

\medskip Now, for existence issues, we observe that problem \equ{l1}-\equ{l3} can be set up in variational form in the Hilbert space
\[
\mathcal H = \Big\{ \phi \in {\mathcal D}^{1,2}(\R^3)\;:\;\ix{\phi\,Z_{ij}\,\chi_i} = 0\quad\forall\;i=1,\,2\ldots N,\;j=1,\,2,\,3 \Big\}
\]
endowed with the inner product $\langle\phi,\psi\rangle = \ix{\nabla\psi\cdot\nabla\phi}$, as
\be{Eq:LinearVar}
\ix{\nabla \phi\cdot\nabla \psi} - \ix{\sum_{i=1}^N p\(W_\xi-\la_i + \tfrac12\,\omega^2\,|x'|^2\)^{p-1}_+\,\chi_i\,\phi\,\psi} + \ix{\psi\,h} =0
\ee
for all $\psi \in \mathcal H$. Since the potential defined by the second term of the above equality is compactly supported and $h$ decays sufficiently fast, this equation takes the form $\phi + \mathsf K[\phi] =\tilde h$ where $\mathsf K$ is a compact linear operator of $\mathcal H$. The equation for $\tilde h = 0$ has just the trivial solution in view of estimate \equ{cota}. Fredholm's alternative thus applies to yield existence. This concludes the proof of Lemma~\ref{linear}.\qed\end{proof}

We conclude this section with some considerations on the differentiability of the solution with respect to the parameter $\xi$. Let us assume that $h= h(\cdot,\xi)$ defines a continuous operator into the space of functions with finite $\|\cdot \|_{**}$-norm. We also assume that $\| \partial_\xi h(\cdot,\xi)\|_{**}<+\infty$. Let $\phi = \phi(\cdot, \xi)$ be the unique solution of Problem \equ{l1}-\equ{l3} for that right hand side, with corresponding constants $c_{ij}(\xi)$. Then $\phi$ is differentiable in $\xi$. Moreover $\partial_\xi \phi$ can be decomposed as
\[
\partial_\xi \phi = \psi + \sum_{ij}\,d_{ij}\,Z_{ij}\,\chi_j
\]
where $\psi$ solves
\[
\begin{array}{c}
\displaystyle\mathsf L[\psi] = \partial_\xi h\;- \sum_{i=1}^N p\,\partial_\xi\Big[\(W_\xi-\la_i+\tfrac12\,\omega^2\,|x'|^2\)^{p-1}_+ \chi_i\Big]\,\phi\hspace*{3.5cm}\\
\hspace*{3.5cm}\displaystyle +\sum_{i=1}^N\sum_{j=1}^3\Big[c_{ij}\,\partial_\xi\(Z_{ij}\,\chi_i\)+b_{ij}\,Z_{ij}\,\chi_i\Big]\,,\\
\displaystyle\lim_{|x|\to\infty}\psi(x)=0\;,\quad\ix{\psi\,Z_{ij}\,\chi_i} = 0\quad\forall\;i=1,\,2\ldots N,\;j=1,\,2,\,3\;,
\end{array}
\]
and the constants $d_{ij}$ are chosen so that $\eta:=\sum_{i=1}^N\sum_{j=1}^3d_{ij}\,Z_{ij}$ satisfies
\[
\ix{\chi_{ij}\,Z_{ij}\,\eta} = - \ix{\partial_{\xi}\(\chi_{ij}\,Z_{ij}\) \phi}\quad\forall\;i=1,\,2\ldots N,\;j=1,\,2,\,3\;.
\]
\begin{lemma}\label{derivative} With the same notations and conditions as in Lemma~\ref{linear}, we have
\[
\|\partial_\xi \phi(\cdot, \xi)\|_*\,\le\,C\,\big(\,\|h (\cdot, \xi) \|_{**} + \|\partial_\xi h(\cdot, \xi)\|_{**} \big)\,.
\]Ê\end{lemma}

\section{The projected nonlinear problem}\label{Sec:Projection}

Next we want to solve a {\em projected version} of the nonlinear problem \equ{ddd} using Lemma~\ref{linear}. Thus we consider the problem of finding $\phi$ with $\|\phi\|_* < +\infty$, solution~of
\be{nl1}
\mathsf L[\phi] = -\mathsf E -\mathsf N[\phi] + \sum_{i=1}^N\sum_{j=1}^3 c_{ij}\,Z_{ij}\,\chi_i
\ee
\be{nl3}
\lim_{|x|\to +\infty}\phi(x)=0
\ee
where the coefficients $c_{ij}$ are Lagrange multipliers associated to the orthogonality conditions
\be{nl2}
\ix{\phi\,Z_{ij}\,\chi_i} = 0\quad\forall\;i=1,\,2\ldots N,\;j=1,\,2,\,3\;.
\ee
In other words, we look for a critical point of the functional $\mathsf J$ defined by \equ{J} under the constraints \equ{nl2}.

For this purpose, we first have to measure the error $\mathsf E$. We recall that
\begin{eqnarray*}
\mathsf E &=&\Delta W_\xi+ \sum_{i=1}^N \(W_\xi- \la_i + \tfrac12\,\omega^2\,|x'|^2 \)^p_+\,\chi_i\\
&=&\sum_{i=1}^N\,\Big[\Big( w_i +\sum_{j\ne i} w_j - \la_i + \tfrac12\,\omega^2\,|x'|^2 \Big)^p_+- (w_i -\la_i)_+^p\,\Big]\,\chi_i\\
&=&\sum_{i=1}^N\,p\,\Big[w_i -\la_i +t\,\Big(\sum_{j\ne i} w_j + \tfrac12\,\omega^2\,|x'|^2 \Big)\Big]^{p-1} _+ \Big(\sum_{j\ne i} w_j + \tfrac12\,\omega^2\,|x'|^2 \Big)\,\chi_i
\end{eqnarray*}
for some function $t$ taking values in $(0,1)$. It follows that
\[
|\mathsf E| \le C \sum_{i=1}^N\,\Big[\sum_{j\ne i} \frac 1 {|\xi_i-\xi_j|} +\frac 12\,\omega^2\,|\xi_i|^2\Big]\,\chi_i\;\le C\,\omega^{ 2/3}\sum_{i=1}^N \chi_i\;,
\]
from which we deduce the estimate
\[
\|\mathsf E\|_{**} \le C\,\omega^{ 2/3}\;.
\]

As for the operator $\mathsf N[\phi]$, we easily check that for $\|\phi\|_* \le 1$,
\[
\big|\,\mathsf N[\phi]\,\big| \le C\,\sum_{i=1}^N |\phi|^\gamma\,\chi_i\quad\mbox{with}\;\gamma:=\min\{p,2\}\;,
\]
which implies
\[
\big\|\,\mathsf N[\phi] \,\big\|_{**} \le C\,\|\phi\|_*^\gamma\;.
\]
Let $\mathsf T$ be the linear operator defined in Lemma \ref{linear}. Equation \equ{nl1} can be rewritten~as
\[
\phi = {\mathsf A} [\phi] := -\mathsf T\big[\mathsf E + \mathsf N[\phi]\big]\;.
\]
Clearly the operator ${\mathsf A}$ applies the region
\[
{\mathcal B} = \left\{ \phi\in L^\infty(\R^3)\;:\;\|\phi\|_* \le K\,\omega^{2/3} \right\}
\]
into itself if the constant $K$ is fixed, large enough. It is straightforward to check that $\mathsf N[\phi]$ satisfies in this region a Lipschitz property of the form
\[
\big\|\,\mathsf N[\phi_1]-\mathsf N[\phi_2] \,\big\|_{**} \le \kappa_\omega\,\|\phi_1 - \phi_2 \|_*
\]
for some positive $\kappa_\omega$ such that $\lim_{\omega\to 0}\kappa_\omega=0$, and hence existence of a unique fixed point $\phi$ of $\mathsf A$ in $\mathcal B$ immediately follows for $\omega$ small enough. We have then solved the projected nonlinear problem.

\medskip Since the error $\mathsf E$ is even with respect to the variable $x^3$, uniqueness of the solution of \equ{nl1}-\equ{nl2} implies that this symmetry is also valid for $\phi$ itself, and besides, the numbers $c_{i3}$ are automatically all zero. Summarizing, we have proven the following result.
\begin{lemma}\label{nonlinear} Assume that $\xi=(\xi_1,\,\xi_2,\ldots\xi_N)\in \R^{2N}$ is given and satisfies~\equ{constraints}. Then Problem \equ{nl1}-\equ{nl2} has a unique solution $\phi_\xi$ which depends continuously on $\xi$ and $\omega$ for the $\|\ \|_*$-norm and satisfies $\|\phi_\xi\|_*\le C\,\omega^{2/3}$ for some positive $C$, which is independent of~$\omega$, small enough. Besides, the numbers $c_{i3}$ are all equal to zero for $i=1,\,2\ldots N$. \end{lemma}
It is important to mention that $\phi_\xi$ also defines a continuously differentiable operator in its parameter.
Indeed, combining its fixed point characterization with the implicit function theorem and the result of Lemma~\ref{derivative},
we find in fact that
\[\label{derivative2}
\|\partial_\xi \phi_\xi \|_* \le C\,\omega^{2/3}\,.
\]
We leave the details to the reader.

\medskip With the complex notation of Section~\ref{Sec:Intro}, let us consider the rotation $e^{i\,\alpha }$ of an angle $\alpha$ around the $x^3$-axis and let $e^{i\,\alpha }\,\xi=(e^{i\,\alpha }\,\xi_1,\ldots e^{i\,\alpha }\,\xi_N)$. By construction, there is a rotational symmetry around the $x^3$-axis, which is reflected at the level of Problem \equ{nl1}-\equ{nl2} as follows.
\begin{lemma} Consider the solution $\phi$ found in Lemma~\ref{nonlinear}. For any $\alpha\in\R$ and any $(x',x^3)\in\C\times\R$, we have that
\[
\phi_{e^{i\,\alpha }\xi}(x',x^3)=\phi_\xi(e^{-i\,\alpha}\,x',x^3)\;.
\]
\end{lemma}
The proof is a direct consequence of uniqueness and rotation invariance of the equation satisfied by $\phi_\xi$.

\section{The variational reduction}\label{Sec:Reduction}

We consider the functional $\mathsf J$ defined in \equ{J}. Our goal is to find a critical point satisfying \equ{nl2}, of the form $u = W_\xi+ \phi_\xi$. We estimate $\mathsf J\,[W_\xi]$ by computing first
\[
\ix{|\nabla W_\xi|^2} =\ix{\Big| \sum_{i=1}^N \nabla w_i \Big |^2\!} = \sum_{i=1}^N\ix{|\nabla w_i |^2} + \sum_{i\ne j} \ix{\nabla w_i\cdot\nabla w_j}\;.
\]
The last term of the right hand side can be estimated by
\begin{eqnarray*}
\ix{\nabla w_i\cdot\nabla w_j} &=& -\ix{\Delta w_i\,w_j} = \ix{(w_i-\la_i)_+^p\,w_j}\\
&=&\ix{\big(w^{\la_i}-\la_i\big)_+^p\,w^{\la_j}(x+ \xi_i -\xi_j)}\\
&=&\ix{\big(w^{\la_i}-\la_i\big)_+^p\,\frac {m_j}{4\pi\,|x+\xi_i-\xi_j|}}\\
&=&\frac {m_i\,m_j}{4\pi\,|\xi_i-\xi_j|}+O(\omega^{4/3})
\end{eqnarray*}
where $m_i = \ix{\big(w^{\la_i} -\la_i\big)_+^p} = m_*\,\la_i^{(3-p)/2}$.
Next we find that
\begin{eqnarray*}
&&\ix{\Big( w_i + \sum_{j\ne i}w_j -\la_i + \tfrac12\,\omega^2\,|x'|^2\Big)^{p+1}_+\,\chi_i}\\
&&=\ix{( w_i -\la_i)^{p+1}_+}\\
&&\qquad+\,(p+1)\ix{( w_i -\la_i)^{p}_+\Big(\sum_{j\ne i}w_j + \tfrac12\,\omega^2\,|x'|^2\Big)}\,+\,O(\omega^{4/3})\\
&&=\ix{( w_i -\la_i)^{p+1}_+} + (p+1)\Big(\sum_{j\ne i} \frac {m_i\,m_j}{4\pi\,|\xi_i-\xi_j|} + \tfrac12\,\omega^2m_i\,|\xi_i|^2\Big)\,+\,O(\omega^{4/3})\;.
\end{eqnarray*}
Let us define
\[
\mathsf e_* :=\frac 12 \ix{|\nabla w|^2}- \frac 1{p+1} \ix{(w -1)^{p+1}_+}\;.
\]
Combining the above estimates, we obtain that
\be{ap1}
\mathsf J\,[W_\xi]\,=\sum_{i=1}^N \la_i^{5-p}\,\mathsf e_*\,-\,\mathcal V_m^\omega (\xi_1,\ldots\xi_N)\,+\,O(\omega^{4/3})\;,
\ee
where $\mathcal V_m^\omega(\xi)=\sum_{i\ne j}\frac {m_i\,m_j} {8\pi\,|\xi_i -\xi_j|} + \tfrac12\,\omega^2 \sum_{i=1}^Nm_i\,|\xi_i|^2$ has been introduced in Section~\ref{Sec:Intro}. Here the $O(\omega^{4/3})$ term is uniform as $\omega\to 0$ on the set of $\xi$ satisfying the constraints $\equ{constraints}$. This approximation is also uniform in the $C^1$ sense. Indeed, we directly check that
\be{ap2}
\nabla_\xi \mathsf J\,[W_\xi] = - \nabla_\xi \mathcal V_m^\omega (\xi) + O(\omega^{4/3})\;.
\ee
According to \equ{Eqn:ScalingOmega}, we have $\mathcal V_m^\omega (\xi)=\omega^{2/3}\,\mathcal V_m(\zeta)$ for $\zeta = \omega^{2/3}\,\xi$. We get a solution of Problem \equ{1}-\equ{2} as soon as all constants $c_{ij}$ are equal to zero in \equ{nl1}.
\begin{lemma}\label{Lem:CriticalLambda} With the above notations, $c_{ij}=0$ for all $i=1,\,2\ldots N,$ $j=1,\,2,\,3$ if and only if $\xi$ is a critical point of the functional $\xi\mapsto \Lambda(\xi):=\mathsf J\,[W_\xi+\phi_\xi]$.\end{lemma}
\begin{proof} We have already noticed in Lemma~\ref{nonlinear} that the numbers $c_{i3}$ are all equal to zero. On the other hand, we have that
\begin{multline}
\partial_{\xi_{ij}}\Lambda =D\mathsf J\,[W_\xi+\phi_\xi]\cdot\partial_{\xi_{ij}}(W_\xi+ \phi_\xi)=\sum_{k,\,\ell} c_{kl} \ix{\partial_{\xi_{ij}}(W_\xi+ \phi_\xi)\,Z_{k\ell}\,\chi_\ell}\\
=c_{ij}\,\left (\,\ix{|Z_{ij}|^2\,\chi_i}\,\right )\,+\,\Big(\sum_{(k,\,\ell) \ne (i,j) } c_{k\ell}\Big)\,O(\omega^{2/3})\;.
\label{pot}\end{multline}
{}From here the assertion of the lemma readily follows, provided that $\omega$ is sufficiently small.
\qed\end{proof}

\begin{remark}\label{remi}
An important observation that follows from the rotation invariance of the equation is the following. Assume that the point $\xi$ is such that $\xi_\ell = (\xi_{\ell 1},0)\neq(0,0)$ for some $\ell\in\{1,\ldots N\}$. Then if
\[
\partial_{\xi_{kj}}\Lambda (\xi) = 0 \foral k=1,\ldots N\,,\;j=1,2, \quad (k,j)\ne (\ell,2)\;,
\]
it follows that actually $\xi$ is a critical point of $\Lambda$. Indeed, differentiating in $\alpha$ the relation
$\Lambda(e^{i\,\alpha }\,\xi) = \Lambda(\xi)$ we get
\[
0=\sum_{k=1}^N \partial_{\xi_k} \Lambda(\xi) \cdot i\,\xi_k\,=\,\partial_{\xi_\ell}\Lambda(\xi) \cdot i\,\xi_\ell\,=\,-\,\xi_{\ell 1}\,\partial_{\xi_{\ell 2}}\Lambda(\xi)\;,
\]
and the result follows. \end{remark}

\section{Proofs of Theorems \ref{teo1}-\ref{teo2}}\label{Sec:Proofs}

Let us consider the solution $\phi_\xi$ of \equ{nl1}-\equ{nl2}, \emph{i.e.} of the problem
\[
\begin{array}{c}\label{nl11}\displaystyle
\mathsf L[\phi_\xi ] = -\mathsf E -\mathsf N[\phi_\xi] + \sum_{i=1}^N\sum_{j=1}^3 c_{ij}(\xi)\,Z_{ij}\,\chi_i\\ \\
\displaystyle\lim_{|x|\to +\infty}\phi_\xi(x)=0
\end{array}
\]
given by Lemma \ref{nonlinear}. We will then get a solution of Problem \equ{1}-\equ{2}, of the desired form $u= W_\xi +\phi_\xi$, inducing the ones for Theorems \ref{teo1} and \ref{teo2}, if we can adjust $\xi$ in such a way that
\[
c_{ij}(\xi) = 0 \quad \foral\;i=1,\,2\ldots N,\;j=1,\,2,\,3\;.
\]
According to Lemma~\ref{Lem:CriticalLambda}, this is equivalent to finding a critical point of the functional
\[
\Lambda(\xi)\;:=\;\mathsf J\,[W_\xi+\phi_\xi]\;.
\]
We expand this functional as follows:
\[
\mathsf J\,[W_\xi] = \mathsf J\,[W_\xi+\phi_\xi] - D\mathsf J\,[W_\xi+\phi_\xi]\cdot\phi_\xi+ \frac 12\int_0^1 D^2\mathsf J\,[W_\xi+ (1-t)\,\phi_\xi]\cdot(\phi_\xi\,,\,\phi_\xi)\;dt\,.
\]
By definition of $\phi_\xi$ we have that $D\mathsf J\,[W_\xi+\phi_\xi]\cdot\phi_\xi=0$. On the other hand, using Lemma~\ref{nonlinear}, we check directly, out of the definition of $\phi_\xi$, that
\[
D^2\mathsf J\,[W_\xi+ (1-t)\,\phi_\xi]\cdot(\phi_\xi\,,\,\phi_\xi) = O(\omega^{4/3})
\]
uniformly on points $\xi_i$ satisfying constraints \equ{constraints}. Hence, from expansion \equ{ap1} we obtain that
\[\label{expansion}
\Lambda(\xi)=\sum_{i=1}^N \la_i^{5-p}\,\mathsf e_*\,-\,\mathcal V_m^\omega (\xi)\,+\,O(\omega^{4/3})\;.
\]

We claim that this expansion also holds in the $C^1$ sense. Let us first observe that
\[
\ix{\mathsf E\,\partial_{\xi} W_\xi}= \nabla_\xi \mathsf J\,[W_\xi]\quad\mbox{and}\quad\partial_{\xi_{ij}} W_{\xi} = Z_{ij}\;.
\]
Then, testing equation \equ{nl1} against $Z_{ij}$, we see that
\[
\ix{\big(\mathsf N[\phi]\,Z_{ij} + \mathsf L[Z_{ij}]\big)\,\phi}= - \ix{\mathsf E\,Z_{ij}}+ \sum_{kl} c_{kl} \ix{Z_{ij}\,Z_{kl}\,\chi_i}\;.
\]
Next we observe that
\[
\big\|\,\mathsf L[Z_{ij}]\,\big\|_{**} = O ( \omega^{2/3})\quad\mbox{and}\quad\ix{Z_{ij}\,Z_{kl}\,\chi_i}= O ( \omega^{2/3}) \quad\hbox{if } (i,j)\ne (k,l)\;.
\]
By lemma~\ref{nonlinear}, $\|\phi_\xi\|_* = O ( \omega^{2/3})$, and so we get
\[
\big(\,1 +O(\omega^{2/3})\,\big)\,c_{ij}\;=\;O ( \omega^{4/3})\,+\,\partial_{\xi_{ij}} \mathsf J\,[W_\xi]\;.
\]
Hence, according to relation \equ{pot}, we obtain
\[
\(\,\mathrm I_{\,3N} +O(\omega^{2/3})\,\)\,\nabla_\xi \Lambda (\xi)\,=\,\nabla_\xi \mathsf J\,[W_\xi]\,+\,O ( \omega^{4/3})
\]
where $\nabla_\xi \mathsf J\,[W_\xi]$ has been computed in \equ{ap2}. Summarizing, we have found that
\[
\nabla_\xi\Lambda(\xi) = -\nabla_\xi \mathcal V_m^\omega (\xi_1,\ldots\xi_N)\,+\,O(\omega^{2/3})\;.
\]
Therefore, setting $\xi = \omega^{2/3}\,\zeta$ with $\zeta= (\zeta_1, \ldots\zeta_N)$ and defining $\Gamma(\zeta):=\Lambda(\xi)$ on $\mathfrak{B}_\mu := \left\{\zeta\in \R^{2N}\,:\,\mbox{\equ{constraints1} holds}\right\}$, we have shown the following result.
\begin{proposition}\label{Prop:Summary} With the above notations, we have that
\[\label{pata}
\begin{array}{c}
\displaystyle\Gamma ( \zeta)=\sum_{i=1}^N \la_i^{5-p}\,\mathsf e_*\,-\,\omega^{2/3}\,\mathcal V_m( \zeta )\,+\,O(\omega^{4/3})\\
\displaystyle\nabla \Gamma ( \zeta) =\,-\,\nabla \mathcal V_m( \zeta )\,+\,O(\omega^{2/3})
\end{array}\]
uniformly on $\zeta$ satisfying \equ{constraints1}. Here the terms $O(\cdot)$ are continuous functions of~$\zeta$ defined on $\mathfrak{B}_\mu$.\end{proposition}

\subsection{Proof of Theorem \ref{teo1}}\label{Sec:Proof1}

If $\mu>0$ is fixed large enough, we have that
\[
\inf_{\mathfrak{B}_\mu } \mathcal V_m\,<\,\inf_{\partial\mathfrak{B}_\mu }\mathcal V_m\;.
\]
Fixing such a $\mu$, we get from Proposition~\ref{Prop:Summary} that, for all sufficiently small $\omega$,
\[
\sup_{ \mathfrak{B}_\mu } \Gamma\,>\,\sup_{ \partial\mathfrak{B}_\mu }\Gamma\;.
\]
so that the functional $\Lambda$ has a maximum value somewhere in $\omega^{2/3}\,\mathfrak{B}_\mu$, which is close to a maximum value of $\mathcal V_m^\omega$. This value is achieved at critical point of $\Lambda$, and hence a solution with the desired features exists. The construction is concluded.\newline\qed

\subsection{Proof of Theorem \ref{Thm:Lagrange}}\label{Sec:Proof-Lagrange}

When $(\zeta_{1},\ldots\zeta_N)$ is a regular polygon with $\zeta_j$ given by~\equ{polygon} and all masses are equal, the system is invariant under the rotation defined~by
\[
x=\underbrace{(x^1,x^2,x^3)}_{\in\R^3}\approx\underbrace{((x^1+i\,x^2),\,x^3)}_{\in\C\times\R}\mapsto\(e^{2\,i\,\pi/N}(x^1+i\,x^2),\,x^3\)=:\mathcal R_N\,x\;.
\]
We can therefore pass to the quotient with respect to this group of invariance and look for solutions $u$ which are invariant under then action of $\mathcal R_N$ and moreover symmetric symmetric with respect to the reflections $(x^1,x^2,x^3)\mapsto(x^1,-x^2,x^3)$ and $(x^1,x^2,x^3)\mapsto(x^1,x^2,-x^3)$. Here we assume that $(\zeta_{1},\ldots\zeta_N)$ is contained in the plane $\{x^3=0\}$ and $\zeta_1=(r,0,0)$. Altogether this amounts to look for critical points of the functional
\[
\mathsf J_1[u] = \frac 12 \int_{\Omega_1}{|\nabla u|^2}\;dx - \frac 1{p+1}\int_{\Omega_1}{\big( u-\la_1 + \tfrac12\,\omega^2\,|x'|^2\big)^{p+1}_+\,\chi_1}\;dx
\]
where $\Omega_1=\{x=(x',x^3)\in\R^3\approx\C\times\R\,:\,x'=r\,e^{i\,\theta}\;\mbox{s.t.}\;-\tfrac\pi N<\theta<\tfrac\pi N\}$ and $u\in H^1(\Omega)$ is invariant under the two above reflections and such that $\nabla u\cdot n=0$ on $\partial\Omega\setminus\{0\}\times\R$ and $\chi_1$ is the characteristic function of the support of $\rho_1=\big( u-\la_1 + \tfrac12\,\omega^2\,|x'|^2\big)^{p+1}_+\,\chi_1$ in $\Omega_1$. Here $n=n(x)$ denotes the unit outgoing normal vector at $x\in\partial\Omega_1$. With $\mathsf J$ defined by \eqref{J}, it is straightforward to see that $\mathsf J\,[u]=N\,\mathsf J_1[u]$ if $u$ is extended to $\R^3$ by assuming that $u(\mathcal R_N\,x)=u(x)$. With these notations, we find that
\[
\mathcal V_m(\zeta_{1},\ldots\zeta_N) = N\,m_*\,\(\frac{a_N}{4\pi}\,\frac{m_*}r +\frac 12\,r^2\)\;.
\]
The proof goes as for Theorem \ref{teo1}. Because of the symmetry assumptions, $c_{1j}=0$ if $j=2$ or $3$. Details are left to the reader. \qed

\subsection{Proof of Theorem \ref{teo2}}\label{Sec:Proof2}

We look for a critical point of the functional $\Gamma$ of Proposition~\ref{Prop:Summary} in a neighborhood of a critical point $\zeta$ of $\mathcal V_m$, which is nondegenerate up to rotations. With no loss of generality, we may assume that $\zeta_1\ne 0$, $\zeta_{12}=0$ and denote by $\tilde{\mathcal V}_m$ the restriction of $\mathcal V_m$ to $(\R\times\{0\})\times(\R^2)^{N-1}\ni\zeta$. Similarly, we denote by $\tilde\Gamma$ the restriction of $\Gamma$ to $(\R\times\{0\})\times(\R^2)^{N-1}$.

By assumption, $\zeta$ is a non-degenerate critical point of $\tilde{\mathcal V}_m$, \emph{i.e.} an isolated zero of $\nabla\tilde{\mathcal V}_m$. Besides, its local degree is non-zero. It follows that on an arbitrarily small neighborhood of that point, the degree for $\nabla \tilde \Gamma$ is non-zero for all sufficiently small $\omega$. Hence there exists a zero $\zeta^\omega\in(\R\times\{0\})\times(\R^2)^{N-1}$ of $\nabla\tilde\Gamma$ as close to~ $\zeta$ as we wish. From the rotation invariance, it follows that $\zeta^\omega$ is also a critical point of $\Gamma$. The proof of Theorem~\ref{teo2} is concluded. \qed

\appendix{} \renewcommand{\thesection}{Appendix \Alph{section}}
\section{Facts on Relative Equilibria}\label{Appendix:RE}

In this Appendix we have collected some results on the $N$-body problem introduced in Section~\ref{Sec:Intro} which are of interest for the proofs of Theorems~\ref{teo1}-\ref{teo2}, with a list of relevant references.

\subsection*{Non-degeneracy of relative equilibria in a standard form}

Relative equilibria are by definition critical points of the function $\mathcal V_m : \R^{2N}\to\R$ defined~by
\[
\mathcal V_m(\zeta_1,\zeta_2,\ldots\zeta_N)= \frac{1}{8\pi}\sum_{i\neq j=1}^N\frac{m_i\,m_j}{|\zeta_j-\zeta_i|} + \frac12\,\sum_{i=1}^Nm_i\,|\zeta_i|^2\,.
\]
Here we assume that $N\ge2$, and $m_i>0$, $i=1,\ldots N$ are given parameters.

\medskip Following Smale in \cite{smale}, we can rewrite this problem as follows. Let us consider the $(2N-3)$-dimensional manifold
\[
S_m:=\Big\{ q= (q_1,\ldots q_N)\in\R^{2N}\,:\,\sum_{i=1}^Nm_i\,(q_i,\tfrac12\,|q_i|^2)=(0,1)\,,\;q_i\neq q_j\hbox{ if }i\neq j \Big\}\,.
\]
The problem of finding critical points of the functional
\[
U_m(q_1,\ldots q_N)= \frac{1}{8\pi}\sum_{i\neq j=1}^N\frac{m_i\,m_j}{|q_j-q_i|}
\]
on $S_m$ is equivalent to that of relative equilibria; see for instance \cite{elmabsout}. Let us give some details. Let $\bar q$ be a critical point of $U_m$ on $S_m$. Then by definition, there are Lagrange multipliers $\la\in\R$ and $\mu\in\R^2$ for which
\[\label{qq}
-\frac{1}{8\pi }\sum_{i\neq j}^N\frac{m_i\,m_j}{|\bar q_i-\bar q_j|^3}\,(\bar q_i-\bar q_j)\,=\lambda\,m_j\,\bar q_j + m_j\,\mu\quad \forall\;j=1,\ldots N\;.
\]
First, adding in $j$ the above relations and using that $M=\sum_{j=1}^N m_j >0$ we obtain that $\mu=0$. Second, taking the scalar product of $\R^2$ against $\bar q_j$ and then adding in $j$, we easily obtain that $U_m(\bar q) = \la$. From here it follows that the point $\bar \zeta = \la^{1/3}\,\bar q$ is a critical point of the functional $\mathcal V_m$, hence a relative equilibrium.

With the reparametrization of $\R^{2N}$ given by
\[\label{coord}
\zeta (\alpha,p,q) = (\zeta_1,\ldots \zeta_N) = (\alpha\,q_1 + p,\ldots \alpha\,q_N + p )\;,\quad (\alpha,p,q)\in \R\times\R^2\times S_m\;,
\]
the Hessian matrix of $\mathcal V_m$ at the critical point $\bar\zeta=\zeta(\la^{1/3},0,\bar q)$ found above is represented as the block matrix
\[\label{cord}
D^2\mathcal V_m(\bar\zeta)=\left( \begin{array} {rrc}
2 & & \\
& M\,{\mathrm I_2} & \\
& & \la^{-1/3}\,D^2_{ S_m}\!U_m(\bar q)
\end{array}\right)
\]
where ${\mathrm I_2}$ is the $2\times 2$ identity matrix and $D^2_{ S_m}$ represents the second covariant derivative on $S_m$. Reciprocally, we check that a critical point $\bar \zeta=(\zeta_j)_{j=1}^N$ of $\mathcal V_m$ necessarily satisfies $\sum_{j=1}^N\,m_j\,\zeta_j =0$. Defining $
\bar q = \big(\frac 12\sum_{j=1}^N m_j\,|\zeta_j|^2\big)^{-1/2}\,\bar \zeta$, we readily check that $\bar q$ is a critical point of $U_m$ in $S_m$.

\bigskip Any rotation $e^{i\,\alpha}\,\bar q$ of a critical point $\bar q$ of $U_m$ on $S_m$ is also a critical point. We say that two such critical points are {\em equivalent} in $S_m$. Let us denote by~$\mathcal S_m$ the quotient manifold of $S_m$ by this equivalence relation. On $\mathcal S_m$, critical points of the potential $U_m$ yield critical points of $U_m$ on $S_m$ and hence equivalence classes of critical points $e^{i\,\alpha}\,\zeta$ for~$\mathcal V_m$ using the reparametrization.

A critical point $\tilde q$ of $U_m$ on $\mathcal S_m$ is said to be {\em non-degenerate} if the second variation of $U_m$ at $\tilde q$ is non-singular. Let us assume that $\tilde q_\ell \ne 0$, with either $\ell=1$, or $\ell=2$ if $\tilde q_1=0$. Then there is a unique representative $\bar q$ of this class of equivalence for which $\bar q_{\ell 2} =0$. It is a routine verification to check that $\bar q$ is then a critical point of $U_m$ on the $(2N-4)$-dimensional manifold
\[
\mathsf S_m := \big\{ q\in S_m\;:\;q_\ell \ne 0\;\mbox{as above}\,,\; q_{\ell 2} =0\big\}\,.
\]
Moreover, the second derivative of $U_m$ on $\mathcal S_m$ at $\tilde q$ is non-degenerate if and only if $D^2_{\mathsf S_m}\!U_m (\bar q)$ is non-singular. Because of the expression of $D^2\mathcal V_m(\bar\zeta)$, we see that $\bar \zeta$ is non-degenerate as a critical point of $\mathcal V_m$ on the space of $\zeta\in \R^{2N}$ with $q_{\ell 2} = 0$, which is the notion of \emph{non-degeneracy up to rotations} of a relative equilibrium that we have used in this paper. Finally we define the {\em index} of a non-degenerate relative equilibrium $\bar \zeta$ as the number of negative eigenvalues of $D^2_{\mathsf S_m}\!U_m (\bar q) $.

\subsection*{Some results on classification of relative equilibria}

For simplicity, we will assume that masses are all different: for any $i$, $j=1,\,\ldots N$, if $m_i=m_j$, then $i=j$. This is the generic case.

The cases $N=2$, $3$ are well known; see for instance \cite{meyeretal}. For $N=2$, the only class of critical points is such that
\[
|\zeta_1-\zeta_2|=\(\tfrac M{4\pi}\)^{1/3}\quad\mbox{and}\quad m_1\,\zeta_1+m_2\,\zeta_2=0\quad\mbox{with}\;M=m_1+m_2\;.
\]
For $N=3$, there are two types of solutions, the Lagrange and the Euler solutions. The \emph{Lagrange solutions} are such that their center of mass is fixed at the origin, the masses are located at the vertices of an equilateral triangle, and the distance between each point is $(M/(4\pi) )^{1/3}$ with $M=m_1+m_2+m_3$. They give rise to two classes of solutions corresponding to the two orientations of the triangle when labeled by the masses. The \emph{Euler solutions} are made of aligned points and provide three classes of critical points, one for each ordering of the masses on the line.

In the case $N\geq 4$, the classes of solutions for which all points are collinear still exist (see \cite{MR1503509}) and are known as the \emph{Moulton solutions}. But the configuration of relative equilibria where all particles are located at the vertices of a regular $N$-polygon exists if and only if all masses are equal; see \cite{macmillan,williams,perkowalter,elmabsout,xiezhang}. Various classification results which have been obtained by Palmore are summarized below.
\begin{theorem}[\cite{MR0363076,MR0413647,MR0413646,MR0445539,MR0420713}]\label{palmore} We have the following multiplicity results.
\begin{enumerate}
\item[{\rm (a)}] For $N\geq 3$, the index of a relative equilibrium is always greater or equal than $N-2$. This bound is achieved by Moulton's solutions.
\item[{\rm (b)}] For $N\geq 3$, there are at least $\mu_i(N):=\binom Ni (N-1-i)\,(N-2)\,!$ distinct relative equilibria in $\mathcal S_m$ of index $2N-4-i$ if $U_m$ is a Morse function. As a consequence, there are at least
\[
\sum_{i=0}^{N-2}\mu_i(N)=[2^{N-1}(N-2)+1]\,(N-2)\,!
\]
distinct relative equilibria in $\mathcal S_m$ if $U_m$ is a Morse function.
\item[{\rm (c)}] For every $N\geq 3$ and for almost all masses $m\in\R^N_+$, $U_m$ is a Morse function.
\item[{\rm (d)}] There are only finitely many classes of relative equilibria for every $N\geq 3$ and for almost all masses $m=(m_i)_{i=1}^N\in\R^N_+$.
\end{enumerate}
\end{theorem}

\par\bigskip\noindent{\small{\sl Acknowledgments.\/} This project is part of the scientific program of the MathAmSud network \emph{NAPDE} and of the ANR funded research project\emph{CBDif-Fr}. It has also been supported by grants Fondecyt 1070389, Fondo Basal CMM and Anillo ACT 125.
\newline\scriptsize{\copyright~2010 by the authors. This paper may be reproduced, in its entirety, for non-commercial~purposes.}}

\def\cprime{$'$}

\end{document}